\documentclass[reqno]{amsart}
\usepackage[utf8]{inputenc}
\usepackage{amsmath}
\usepackage{tikz}
\usetikzlibrary{shapes.geometric}
\usepackage{amsthm}
\usepackage{amsfonts}
\usepackage{caption}
\usepackage{amssymb}
\usepackage{verbatim}
\usepackage{subcaption}
\usepackage{xparse}
\usepackage{changepage}
\newtheorem{theorem}{Theorem}
\newtheorem{lemma}[theorem]{Lemma}

\newtheorem{definition}[theorem]{Definition}
\numberwithin{equation}{section}

\title{Ducci on $\mathbb{Z}_m^n$ and the Maximum Length for $n$ Odd}
\author{Mark L. Lewis}
\address{Department of Mathematical Sciences\\
Kent State University\\
Kent, OH 44242}
\email{lewis@math.kent.edu}
\author{Shannon M. Tefft}
\address{Department of Mathematical Sciences\\
Kent State University\\
Kent, OH 44242}
\email{stefft@kent.edu}
\date{August 2024}


\begin{document}
\begin{abstract}
Define the Ducci function $D: \mathbb{Z}_m^n \to \mathbb{Z}_m^n$ so 
\[D(x_1,x_2, ...,x_n)=(x_1+x_2 \;\text{mod} \; m, x_2+x_3 \; \text{mod} \; m, ..., x_n+x_1 \; \text{mod} \; m).\]
 Call $\{D^{\alpha}(\mathbf{u})\}_{\alpha=0}^{\infty}$ the Ducci sequence of $\mathbf{u}$. Because $\mathbb{Z}_m^n$ is finite, every Ducci sequence will enter a cycle. In this paper, we will prove that if $n$ is odd and $m=2^lm_1$ where $m_1$ is odd, then the longest it will take for a Ducci sequence to enter its cycle is $l$ iterations. Furthermore, we will prove the set of all tuples in a cycle for $\mathbb{Z}_m^n$ is $\{(x_1, x_2, ..., x_n) \in \mathbb{Z}_m^n \; \mid \; x_1+x_2+ \cdots +x_n \equiv 0 \; \text{mod} \; 2^l\}$.
\end{abstract}

\maketitle

\section{Introduction}\label{Intro}
 Define a function $D: \mathbb{Z}_m^n \to \mathbb{Z}_m^n$, as
\[D(x_1, x_2, ..., x_n)=(x_1+x_2 \; \text{mod} \; m, x_2+x_3 \; \text{mod} \; m, ..., x_n+x_1 \; \text{mod} \; m).\]
Like in \cite{Breuer1, Ehrlich, Glaser}, we call $D$ the Ducci function and the sequence $\{D^{\alpha}(\mathbf{u})\}_{\alpha=0}^{\infty}$ the \textbf{Ducci sequence of} $\mathbf{u}$ for $\mathbf{u} \in \mathbb{Z}_m^n$.

\indent To illustrate what a Ducci sequence looks like, we can look at $(3,0,3) \in \mathbb{Z}_4^3$ and the first eight terms in its Ducci sequence: $(3,0,3),(3,3,2),$ $(2,1,1),$ $(3,2,3),$ $(1,1,2),$  $(2,3,3),(1,2,1),(3,3,2)$. If we continue listing terms of the sequence, we then cycle through the tuples $(3,3,2),$ $(2,1,1),(3,2,3),$ $(1,1,2),$ $(2,3,3),$ $(1,2,1)$. These repeating tuples are the Ducci cycle of $(3,0,3)$. To give a more formal definition, 

\begin{definition}
The \textbf{Ducci cycle of} $\mathbf{u}$ is  
\[\{\mathbf{v} \mid \exists \alpha \in \mathbb{Z}^+ \cup \{0\}, \beta \in \mathbb{Z}^+  \ni \mathbf{v}=D^{\alpha+\beta}(\mathbf{u})=D^{\alpha}(\mathbf{u})\}\].
The \textbf{length of} $\mathbf{u}$, $\mathbf{Len(u)}$, is the smallest $\alpha$ satisfying the equation 
\[\mathbf{v}=D^{\alpha+\beta}(\mathbf{u})=D^{\alpha}(\mathbf{u})\]
 for some $v \in \mathbb{Z}_m^n$ \and the \textbf{period of} $\mathbf{u}$, $\mathbf{Per(u)}$, is the smallest $\beta$ that satisfies this equation. 
\end{definition}

If for a tuple $\mathbf{v} \in \mathbb{Z}_m^n$, there exists $\mathbf{u} \in \mathbb{Z}_m^n$ such that $\mathbf{v}$ is in the Ducci cycle of $\mathbf{u}$, we may also say that $\mathbf{v}$ is in a Ducci cycle. Note that because $|\mathbb{Z}_m^n| < \infty$, every Ducci sequence must enter a cycle, which is also stated on page 6000 of \cite{Breuer1}, as well as \cite{Breuer2, Dular}.

\indent When it comes to studying Ducci sequences, $(0,0,...,0,1) \in \mathbb{Z}_m^n$, and namely its Ducci sequence are significant. The Ducci sequence of $(0,0,...,0,1) \in \mathbb{Z}_m^n$ is known as the \textbf{basic Ducci sequence} of $\mathbb{Z}_m^n$.
This term is first used by \cite{Ehrlich} on page 302, and is also used in  \cite{Breuer1, Glaser, Paper1}.
We say 
\[L_m(n)=Len(0,0,...,0,1)\] and \[P_m(n)=Per(0,0,...,0,1)\]
 like in \cite{Paper1} and similar to how it is defined in \cite{Breuer1, Dular,Ehrlich}. These values, $L_m(n)$ and $P_m(n)$ are important because for $\mathbf{u} \in \mathbb{Z}_m^n$, $Len(\mathbf{u}) \leq L_m(n)$ and $Per(\mathbf{u})|P_m(n)$, as proved by \cite{Breuer1} in Lemma 1. 
Therefore, this finding provides a maximum value for the length and period of any tuple in $\mathbb{Z}_m^n$. The notation $P_m(n)$ is also used to represent the maximal period for a Ducci sequence on $\mathbb{Z}_m^n$ on page 858 of \cite{Breuer2}. 

\indent The goal of this paper is to establish a value for $L_m(n)$ when $n$ is odd. Specifically,
\begin{theorem}\label{n_odd_length}
    Let $n$ be odd and let $m=2^lm_1$ where $m_1$ is odd, $l\geq 0$. Then $L_m(n)=l$.
\end{theorem}

\indent Let $K(\mathbb{Z}_m^n)=\{\mathbf{u} \in \mathbb{Z}_m^n \mid \mathbf{u} \; \textit{is in the Ducci cycle for some tuple} \; \mathbf{v} \in \mathbb{Z}_m^n\}$, which is first defined in Definition 4 of \cite{Breuer1}. It is also stated in \cite{Breuer1} that $K(\mathbb{Z}_m^n)$ is a subgroup of $\mathbb{Z}_m^n$ and is Theorem 1 of \cite{Paper1}, where a proof is given.
We plan to use Theorem \ref{n_odd_length} to prove the following theorem about $K(\mathbb{Z}_m^n)$:
\begin{theorem}\label{n_odd_Cycle_subgroup}
     Let $n$ be odd and $m=2^lm_1$ where $m_1$ is odd, $l\geq 1$. Then 
     \[K(\mathbb{Z}_m^n)=\{(x_1, x_2, ..., x_n) \in \mathbb{Z}_m^n \; \mid \; x_1+x_2+ \cdots +x_n \equiv 0 \; \text{mod} \; 2^l\}\].
\end{theorem}
\indent The work in this paper was done while the second author was a Ph.D. student at Kent State University under the advisement of the first author and will appear as part of the second author's dissertation. The authors have no competing interests.

\section{Background}\label{Background}
\indent The Ducci function has been defined in a few different ways in the literature. For example, \cite{Ehrlich, Freedman, Glaser, Furno} define the Ducci function $\bar{D}$ to be an endomorphism of $(\mathbb{Z}^+ \cup \{0\})^n$, \cite{Breuer1} defines it as an endomorphism on $\mathbb{Z}^n$, and \cite{Brown, Chamberland, Schinzel} define it as an endomorphism on $\mathbb{R}^n$, where $\bar{D}(x_1, x_2, ..., x_n)=(|x_1-x_2|, |x_2-x_3|, ..., |x_n-x_1|)$ for all three of these cases. If $\mathbf{u} \in \mathbb{Z}^n$, then $\bar{D}(\mathbf{u}) \in (\mathbb{Z}^+ \cup \{0\})^n$. Because of this, for the sake of simplicity, we will refer to the Ducci cases on $\mathbb{Z}^n$ and on $(\mathbb{Z}^+ \cup \{0\})^n$ as the Ducci case on $\mathbb{Z}^n$. As expected, Ducci on $\mathbb{Z}^n$ and Ducci on $\mathbb{R}^n$ need to be considered separately, and findings on one case may not apply to the other.
 
\indent In the typical Ducci case on $\mathbb{Z}^n$, \cite{Ehrlich, Glaser, Furno} discuss how all Ducci sequences enter a Ducci cycle and \cite{Schinzel} proves this in Theorem 2 for Ducci on $\mathbb{R}^n$. In Lemma 3, \cite{Furno} proves that if you examine a tuple in a Ducci cycle for Ducci on $\mathbb{Z}^n$, then all of the coordinates of the tuple belong to $\{0,c\}$ for some $c \in \mathbb{Z}^+$. As a result of this, the Ducci case on $\mathbb{Z}^n$ often focuses on the applications of Ducci on $\mathbb{Z}_2$ and how it applies to Ducci on $\mathbb{Z}^n$, as pointed out in \cite{Breuer1, Ehrlich, Glaser, Furno}. Theorem 1 of \cite{Schinzel} proves a similar result for Ducci on $\mathbb{R}^n$: all of the coordinates of a tuple in the sequence belong to $\{0,c\}$ for some $c \in \mathbb{R}$ after the sequence reaches a limit point. 

\indent There have been a few papers that examine the maximum value for the length on the general Ducci case through the Ducci case on $\mathbb{Z}_2^n$. For the $n$ odd case, \cite{Ehrlich} shows on page 303 that $L_2(n)=1$, which supports Theorem \ref{n_odd_length}. For the case where $n$ is even, Theorem 6 of \cite{Glaser} shows if $n=2^{k_1}+2^{k_2}$ where $k_1>k_2 \geq 0$, then $L_2(n)=2^{k_2}$. This is more generally approached by \cite{Breuernote} in Theorem 4, who proves $L_2(n)=2^k$ when $n=2^kn_1$ for $n_1$ odd and $k \geq 1$. 

\indent We focus our attention again on the Ducci function $D$ defined on $\mathbb{Z}_m^n$, which is originally explored in \cite{Wong} and also by \cite{Breuer1, Breuer2, Dular, Paper1}.
$D$ is an endomorphism on $\mathbb{Z}_m^n$, which is a result of how it is defined in Definition 1 of \cite{Breuer1} and is proved on page 4 of \cite{Paper1}. Let $H: \mathbb{Z}_m^n \to \mathbb{Z}_m^n$, such that $H(x_1, x_2, ..., x_n)=(x_2, x_3, ..., x_n, x_1)$. Then by Definition 1 in \cite{Breuer1} and \cite{Ehrlich, Glaser, Paper1}, $H \in End(\mathbb{Z}_m^n)$, $D$ and $H$ commute, and $D=I+H$ where $I$ is the identity endomorphism on $\mathbb{Z}_m^n$. Also, $\{H^{\beta}(D^{\alpha}(\mathbf{u}))\}_{\alpha=0}^{\infty}$ is the Ducci sequence for $H^{\beta}(\mathbf{u})$ and  $H^{\beta}(\mathbf{u}) \in K(\mathbb{Z}_m^n)$ if $\mathbf{u} \in K(\mathbb{Z}_m^n)$ where $0 \leq \alpha \leq n-1$, as it is shown on page 5 of \cite{Paper1}. Like in the Ducci case on $\mathbb{Z}^n$, $D(\lambda \mathbf{u})=\lambda D(\mathbf{u})$, so the Ducci sequence of $D(\lambda \mathbf{u})$ is $\{\lambda D^{\alpha}(\mathbf{u})\}_{\alpha=0}^{\infty}$.

\indent Let $\mathbf{u} \in \mathbb{Z}_m^n$. If there exists $\mathbf{v} \in \mathbb{Z}_m^n$ such that $D(\mathbf{v})=\mathbf{u}$, then $\mathbf{v}$ is known as a \textbf{predecessor} of $\mathbf{u}$. The first instance we can find this definition being used is on page 313 of \cite{Furno} and it is also used by \cite{Breuer1, Glaser}.

\indent We return to our example of $(3,0,3) \in \mathbb{Z}_4^3$. To further examine this example, we create a transition graph that maps out all of the Ducci sequences and their cycles, and then look at the connected component containing $(3,0,3)$, given in Figure \ref{Example}.

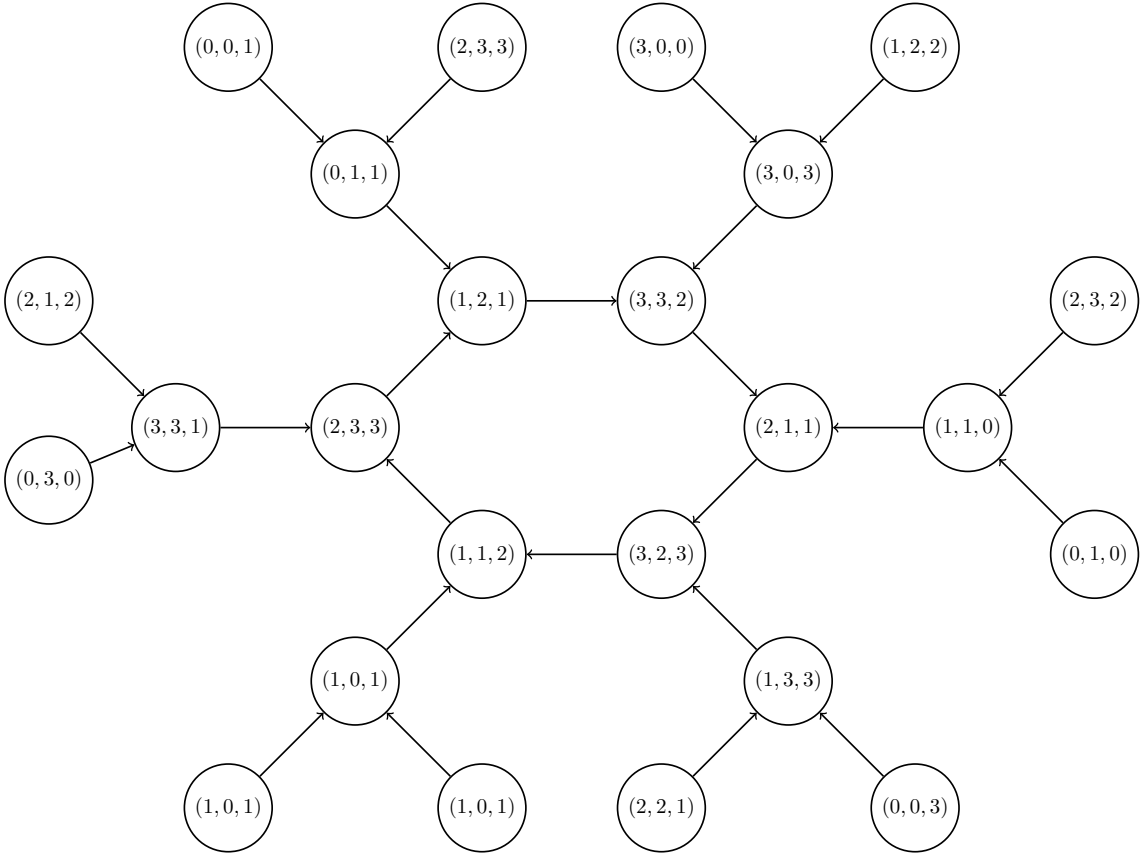
\begin{figure}
\centering
    
\begin{adjustwidth}{-25 pt}{-25 pt}
\resizebox{1.2\textwidth}{!}{
\begin{tikzpicture}[node distance={30mm}, thick, main/.style = {draw, circle}]
\node[main](1){$(1,2,1)$};
\node[main](2)[above left of=1]{$(0,1,1)$};
\node[main](3)[above right of=2]{$(2,3,3)$};
\node[main](4)[above left of=2]{$(0,0,1)$};
\node[main](5)[right of=1]{$(3,3,2)$};
\node[main](6)[above right of=5]{$(3,0,3)$};
\node[main](7)[above left of=6]{$(3,0,0)$};
\node[main](8)[above right of=6]{$(1,2,2)$};
\node[main](9)[below right of=5]{$(2,1,1)$};
\node[main](10)[right of=9]{$(1,1,0)$};
\node[main](11)[above right of=10]{$(2,3,2)$};
\node[main](12)[below right of=10]{$(0,1,0)$};
\node[main](13)[below left of=9]{$(3,2,3)$};
\node[main](14)[below right of=13]{$(1,3,3)$};
\node[main](15)[below left of=14]{$(2,2,1)$};
\node[main](16)[below right of =14]{$(0,0,3)$};
\node[main](17)[left of=13]{$(1,1,2)$};
\node[main](18)[below left of=17]{$(1,0,1)$};
\node[main](19)[below left of=18]{$(1,0,1)$};
\node[main](20)[below right of=18]{$(1,0,1)$};
\node[main](21)[above left of=17]{$(2,3,3)$};
\node[main](22)[left of=21]{$(3,3,1)$};
\node[main](23)[above left of=22]{$(2,1,2)$};
\node[main](24)[below of=23]{$(0,3,0)$};

\draw[->](3)--(2);
\draw[->](4)--(2);
\draw[->](2)--(1);
\draw[->](1)--(5);
\draw[->](6)--(5);
\draw[->](7)--(6);
\draw[->](8)--(6);
\draw[->](5)--(9);
\draw[->](10)--(9);
\draw[->](11)--(10);
\draw[->](12)--(10);
\draw[->](9)--(13);
\draw[->](14)--(13);
\draw[->](15)--(14);
\draw[->](16)--(14);
\draw[->](13)--(17);
\draw[->](18)--(17);
\draw[->](19)--(18);
\draw[->](20)--(18);
\draw[->](17)--(21);
\draw[->](22)--(21);
\draw[->](23)--(22);
\draw[->](24)--(22);
\draw[->](21)--(1);
    
\end{tikzpicture}
}

\end{adjustwidth}
\caption{Transition Graph for $\mathbb{Z}_4^3$}\label{Example}
\end{figure}
Note that this component also contains $(0,0,1)$, so the basic Ducci sequence is a part of this component. Using some of our definitions from Section \ref{Intro}, we can see that $Len(3,0,3)=1$ and $L_4(3)=2$. 

\indent We can also see that $(3,0,3)$ has $2$ predecessors, $(3,0,0)$ and $(1,2,2)$. In fact, all of the tuples in this connected component that have a predecessor have exactly two predecessors. 
When $n$ is odd and $m$ is even, this is always true:
\begin{theorem}\label{n odd m even pred}
For $n$ odd and $m$ even, every tuple that has a predecessor has exactly 2 predecessors. If one predecessor is $(x_1, x_2, ..., x_n)$, then the other predecessor is $(\displaystyle{\frac{m}{2}}+x_1,\displaystyle{\frac{m}{2}}+x_2,..., \frac{m}{2}+x_n)$.
\end{theorem}
\begin{proof}
Let $n$ be odd and $m$ be even. Notice
\[D(x_1, x_2, ..., x_n)=(x_1+x_2,x_2+x_3,...,x_n+x_1)\]
and
\[D(\frac{m}{2}+x_1,\frac{m}{2}+x_2,..., \frac{m}{2}+x_n)=(x_1+x_2,x_2+x_3,...,x_n+x_1.)\]
So if $(x_1, x_2, ..., x_n)$ is a predecessor to a tuple, $(\displaystyle{\frac{m}{2}+x_1,\frac{m}{2}+x_2,..., \frac{m}{2}+x_n})$ is a predecessor to that same tuple. Next we prove that if a tuple has a predecessor, it has exactly 2 predecessors.

\indent Suppose $\mathbf{u}$ has 2 predecessors $(x_1, x_2, ..., x_n)$ and $(y_1,y_2,...,y_n)$. Then we have $x_1+x_2 \equiv y_1+y_2 \; \text{mod} \; m,\; x_2+x_3 \equiv y_2+y_3 \; \text{mod} \; m \;, ...,  x_1+x_n \equiv y_1+y_n \; \text{mod} \; m$. Since the $x_i$ and $y_j$ are each at most $m-1$ for every $1 \leq i,j \leq n$, we obtain 
\[x_1+x_2=y_1+y_2+ z_1m\]
\[x_2+x_3=y_2+y_1+ z_2m\]
\[\vdots\]
\[x_1+x_n=y_1+y_n+ z_nm,\]
where each of the $z_i \in \{-1,0,1\}$ for $1 \leq i \leq n$. Subtracting the second equation from the first yields $x_1-x_3=y_1-y_3+(z_1-z_2)m$. Adding this to the third equation and continuing this pattern, we have 
\[x_1+x_4=y_1+y_4+(z_1-z_2+z_3)m\]
\[\vdots\]
\[x_1-x_n=y_1-y_n+(z_1-z_2+ \cdots -z_{n-1})m.\]
Now if we add this to the equation $ x_1+x_n= y_1+y_n + z_nm$, this produces
\[2x_1=2y_1+(z_1-z_2+ \cdots +z_n)m.\]
We therefore have 2 possible cases:
\begin{itemize}
    \item Case 1: $z_1-z_2+ \cdots +z_n$ is even\\
    Here $x_1=y_1+\delta m$ where $\delta \in \mathbb{Z}$ and therefore $x_1=y_1$.\\
    \item Case 2: $z_1-z_2+ \cdots +z_n$ is odd\\
    Here, $x_1=y_1+\gamma m+\displaystyle{\frac{m}{2}}$ where $\gamma \in \mathbb{Z}$. Therefore, $x_1=y_1+\displaystyle{\frac{m}{2}}$, which is the previously discussed case of this theorem.\\
 
\end{itemize}
\indent If $x_1=y_1$, then 
\[x_2=y_2\] \[x_3=y_3\] \[\vdots\] \[x_n=y_n\].
 If $x_1=y_1+\displaystyle{\frac{m}{2}}$, then
  \[x_2=y_2+\frac{m}{2}\] \[x_3+ \frac{m}{2}=y_3\]\[\vdots\] \[x_n=y_n+\frac{m}{2}\].
\indent The Theorem follows from here.
\end{proof}

\indent Something else to note about the connected component containing $(3,0,3)$ in Figure \ref{Example} is that for every tuple in the cycle, the structure of the branches coming off of these tuples are the same: they each have one predecessor outside of the cycle, and that particular tuple outside the cycle has two predecessors of its own. Part of this is a result of every tuple having a predecessor in $\mathbb{Z}_m^n$ having exactly two predecessors, but if $\mathbf{u}$ is in the cycle of this connected component, then there exists $\mathbf{v} \not \in K(\mathbb{Z}_4^3)$ in the component such that $D^2(\mathbf{v})=\mathbf{u}$ and $D(\mathbf{v}) \not \in K(\mathbb{Z}_m^n)$. Notice that all of these tuples $\mathbf{v}=(x_1, x_2, ..., x_n)$ in this situation satisfy $x_1+x_2+ \cdots+x_n$ is odd. For the case where $n$ is odd and $m$ is even, we will prove the following in Section \ref{MainFinding}:
\begin{lemma}\label{AlltheWayOut}
   Let $n$ be odd and $m=2^lm_1$ with $m_1$ odd and $l \geq 1$. Let $(x_1, x_2, ..., x_n)$ be in $\mathbb{Z}_m^n$ such that $x_1+x_2+ \cdots+ x_n$ is odd. Then $Len(x_1, x_2, ..., x_n)=l$. 
\end{lemma}


 \indent A remaining question that will be relevant to us is this: are there any conditions that a tuple must meet to have a predecessor?
 The next theorem addresses this when $n$ is odd and $m$ is even:
\begin{theorem}\label{whenhaspred}
Let $n$ be odd and $m$ be even. If $(x_1, x_2, ..., x_n) \in \mathbb{Z}_m^n$ has a predecessor, then  $x_1+x_2+\cdots + x_n$ is even. Additionally, all tuples satisfying $x_1+x_2+\cdots + x_n$ even have a predecessor.
\end{theorem}
Notice that in our transition graph from $\mathbb{Z}_4^3$ in Figure \ref{Example}, all of the tuples that have a predecessor satisfy this condition. We would also like to note that for the case where $m$ is prime, this theorem follows from Lemma 4 of \cite{Breuer1}.
\begin{proof}[Proof of Theorem \ref{whenhaspred}]
     Suppose that $(x_1, x_2, ..., x_n)$ has a predecessor $(y_1, y_2,...,y_n)$. Then we have 
\[y_1+y_2\equiv x_1 \; \text{mod} \; m\]
\[y_2+y_3 \equiv x_2 \; \text{mod} \; m\]
\[\vdots\]
\[y_n+y_1 \equiv x_n \; \text{mod} \; m.\]
Adding all of these equations together produces 
\begin{equation}\label{Equation_evenpred_iseven}
2y_1+2y_2+\cdots +2y_n \equiv x_1+x_2+ \cdots + x_n \; \text{mod} \; m.
\end{equation}
Because the left side of (\ref{Equation_evenpred_iseven}) is even and $m$ is even, this forces $x_1+x_2+ \cdots + x_n$ to be even.

\indent It now suffices to show that if a tuple satisfies $x_1+x_2+\cdots + x_n$ is even, then it has a predecessor. We start by counting how many tuples $(x_1, x_2, ..., x_n)$ in $\mathbb{Z}_m^n$ satisfy $x_1+x_2+ \cdots + x_n$ even. The types of tuples that satisfy this are
\begin{itemize}
    \item $x_1, x_2, ..., x_n$ are all even. There are $(\displaystyle{\frac{m}{2})^n}$ such tuples\\
    \item 2 of the $x_i$ are odd and the rest are even. There are $(\displaystyle{\frac{m}{2})^n \times \binom{n}{2}}$ such tuples\\
    \item 4 of the $x_i$ are odd and the rest are even. There are $(\displaystyle{\frac{m}{2})^n \times \binom{n}{4}}$ such tuples\\
    \vdots\\
    \item 1 of the $x_i$ are even and the other $n-1$ are odd. There are $(\displaystyle{\frac{m}{2})^n \times \binom{n}{n-1}}$ such tuples
\end{itemize}
All together, the number of tuples whose entries sum to an even number is

\[\sum_{k=0}^{\frac{n-1}{2}}(\frac{m}{2})^n \binom{n}{2k}=(\frac{m}{2})^n\sum_{k=0}^{\frac{n-1}{2}} \binom{n}{2k}. \]
It is known that $\displaystyle{\sum_{k=0}^{\frac{n-1}{2}} \binom{n}{2k}}=2^{n-1}$, with Identity 129 of \cite{Benjamin} providing a proof. Therefore, this is
\[(\frac{m}{2})^n \times 2^{n-1}\]
or
\[\frac{m^n}{2}.\]

\indent Since the only tuples that have predecessors are of the form $x_1+x_2+ \cdots + x_n$ even, suppose there exists $(x_1, x_2, ..., x_n) \in \mathbb{Z}_m^n$ such that $x_1+x_2+ \cdots + x_n$ is even and $(x_1, x_2, ..., x_n)$ does not have a predecessor. Then there are at most $(\displaystyle{\frac{m^n}{2})-1}$ tuples that have predecessors. Every tuple that has a predecessor has exactly 2 predecessors by Lemma \ref{n odd m even pred}, so if we use this to count all of the tuples in $\mathbb{Z}_m^n$, there are at most $(\displaystyle{(\frac{m^n}{2})-1)} \times 2<m^n$ tuples in $\mathbb{Z}_m^n$. This contradicts the fact that there are $m^n$ tuples in $\mathbb{Z}_m^n$.


\end{proof}


As for the case where $n,m$ are both odd, if you examine a Ducci sequence, you will notice that every tuple is in a Ducci cycle. Therefore, $L_m(n)=0$, which we can prove is always true:
\begin{theorem}\label{n_odd_Length_0}
For $n$ odd and $m$ odd, $L_m(n)=0$.
\end{theorem}
\begin{proof}
For a given $n,m$, if $\mathbf{u}, \mathbf{v} \in \mathbb{Z}_m^n$ and $\mathbf{u}$ is a predecessor to $\mathbf{v}$, then either $\mathbf{u} \not \in K(\mathbb{Z}_m^n)$ and $Len(\mathbf{u})=Len(\mathbf{v})+1$, or $\mathbf{u} \in K(\mathbb{Z}_m^n)$ and $Len(\mathbf{u})=Len(\mathbf{v})=0$.

\indent Now let $n,m$ be odd and $\mathbf{u}=(\displaystyle{\frac{m+1}{2},\frac{m-1}{2},\frac{m+1}{2},..., \frac{m-1}{2}, \frac{m+1}{2}})$.
Then $D(\mathbf{u})= (0,0,...,0,1)$. 
If $(0,0,..., 0,1) \not \in K(\mathbb{Z}_m^n)$, then $\mathbf{u} \not \in K(\mathbb{Z}_m^n)$, which implies $Len(\mathbf{u})>L_m(n)$. But this contradicts $L_m(n) \geq Len(\mathbf{u})$ for every $\mathbf{u} \in \mathbb{Z}_m^n$. Therefore, $(0,0,...,0,1) \in K(\mathbb{Z}_m^n)$ and $L_m(n)=0$.
\end{proof}

Note that this is the case in Theorem \ref{n_odd_length} where $l=0$. It also means that every tuple is in a cycle and $K(\mathbb{Z}_m^n)=\mathbb{Z}_m^n$, which is also proved in Proposition 6.1 of \cite{Dular}. Additionally, every tuple in $\mathbb{Z}_m^n$ has exactly one predecessor.

\indent It would be very useful to be able to know what $D^r(\mathbf{u})$ looks like given a tuple $\mathbf{u} \in \mathbb{Z}_m^n$ for $r \geq 0$. For a tuple $\mathbf{u}=(x_1, x_2, ..., x_n) \in \mathbb{Z}_m^n$, the first few tuples in its Ducci sequence are
 
 \textbf{\[(x_1, x_2, ..., x_n)\]
     \[(x_1+x_2, x_2+x_3, ..., x_n+x_1)\]
     \[(x_1+2x_2+x_3, x_2+2x_3+x_4, ..., x_n+2x_1+x_2\]
     \[(x_1+3x_2+3x_3+x_4, x_2+3x_3+3x_4+x_5, ..., x_n+3x_1+3x_2+x_3)\]
     \[\vdots\]  }
    Notice that the coefficients on each of the $x_i$ for a given coordinate recur in other coordinates of the tuple. If we say that the coefficient on $x_s$ in the first coordinate of $D^r(\mathbf{u})$ is $a_{r,s}$ for $r \geq 0$, then $a_{r,s}$ is also the coefficient on $x_{s-i+1}$ in the $i$th coordinate of $D^r(\mathbf{u})$, which is shown on page 6 of \cite{Paper1}. It will also be useful to use the fact that $D^r(0,0,...,0,1)=(a_{r,n}, a_{r,n-1}, ..., a_{r,1})$ throughout the rest of the paper.
    
    \indent There is an additional property of the $a_{r,s}$ coefficients that will be useful to us:
    \begin{lemma} \label{coefficients_sum_to_power_of_2}
    \[\sum_{i=1}^n a_{r,i}=2^r.\]
\end{lemma}
\begin{proof}
    We prove this by induction
    
    \textbf{Basis Step} $\mathbf{r=0}$: Follows from the fact that $a_{0,1}=1$ and $a_{0,s}=0$ when $1<s \leq n$.
    
    \textbf{Inductive Step:} Assume that $\displaystyle{\sum_{i=1}^n a_{r-1,i}}=2^{r-1}$. Notice $\displaystyle{\sum_{i=1}^n a_{r,s}}$ is
    \[\sum_{i=1}^n a_{r-1,i}+\sum_{i=1}^n a_{r-1,i-1}\, .\]
    Because the $s$ coordinates in $a_{r,s}$ reduces modulo $n$, this is
    \[\sum_{i=1}^n a_{r-1,s}+\sum_{i=1}^n a_{r-1, s} \, .\]
    Since each of these sums is $2^{r-1}$, we conclude that $\displaystyle{\sum_{i=1}^n a_{r,i}}=2^r$.
  
\end{proof}   

\section{Proving $L_m(n)=l$}\label{MainFinding}
\indent We can now begin proving some of our theorems regarding $L_m(n)$ and $K(\mathbb{Z}_m^n)$ when $n$ is odd. To begin, we prove the following lemma:

\begin{lemma}\label{coefficient_at_period}
Let $n$ be odd and $m=2^lm_1$, where $m_1$ is odd and $l\geq 1$. Let $d$ be a number such that $d>L_m(n)$ and $P_m(n)|d$. Then there exists $z \in  \mathbb{Z}^+$ odd such that
\[a_{d,s} \equiv
    \begin{cases}
    zm_1+1 \; \text{mod} \; m & s=1\\
    zm_1  \; \text{mod} \; m & 1 <s \leq n
    \end{cases}.\]

 \end{lemma}
 \begin{proof}
 We prove this via induction
 
   \textbf{Basis step} $\mathbf{l=1}$: Assume $m_1$ odd and $P_m(n)|d$. We first set out to prove that 
    \[a_{d,s} \equiv
    \begin{cases}
    m_1+1 \; \text{mod} \; 2m_1 & s=1\\
    m_1 \; \text{mod} \; 2m_1 & 1 <s \leq n
    \end{cases}.\]
    
   
   \indent By Propostion 3.1 in \cite{Dular}, if $m^*|m$, then $P_{m^*}(n)|P_m(n)$. Therefore, $P_{m_1}(n)|d$.
    
    \indent Since $L_{m_1}(n)=0$, we know that 
    \[a_{d,s} \equiv
    \begin{cases}
     1 \; \text{mod} \; m_1 & s=1\\
     0 \; \text{mod} \; m_1 & 1 <s \leq n
    \end{cases}.\]
    So for $a_{d,s}$, we have that either $a_{d,s} \equiv m_1 \; \text{mod} \; 2m_1$ or $a_{d,s} \equiv 0 \; \text{mod} \; 2m_1$. 
    We also know that $(0,0,...,0,1)$ has $1$ predecessor in $\mathbb{Z}_{m_1}^n$. Since 
    \[D(\frac{m_1+1}{2}, \frac{m_1-1}{2}, \frac{m_1+1}{2}, ..., \frac{m_1-1}{2}, \frac{m_1+1}{2})=(0,0,...,0,1),\] 
    we have that $\displaystyle{(\frac{m_1+1}{2}, \frac{m_1-1}{2}, \frac{m_1+1}{2}, ..., \frac{m_1-1}{2}, \frac{m_1+1}{2})}$ is the predecessor of $(0,0,...,0,1)$. 
    However, another predecessor of$(0,0,...,0,1)$ is 
    \[(a_{d-1,n},a_{d-1,n-1},...,a_{d-1,1}),\]  so  $a_{d-1,i} \equiv \displaystyle{\frac{m_1 + 1}{2}} \; \text{mod} \; m_1$
    when $i$ is odd and $a_{d-1,i} \equiv \displaystyle{\frac{m_1 - 1}{2}} \; \text{mod} \; m_1$ when $i$ is even.
    
   Therefore, for every $i \neq 1$, $a_{d-1,i}$ is odd and $a_{d-1,i-1}$ is even or the parities are switched. Either way, this tells us that $a_{d,s}=a_{d-1,s}+a_{d-1,s-1}$ must be odd for every $1<s \leq n$. Therefore, we must be in the case where $a_{d,s} \equiv m_1 \; \text{mod} \; 2m_1$.

    \indent Since $a_{d,1}=a_{d-1,1}+a_{d-1,n}$ and the parities of $a_{d-1,1}$ and $a_{d-1,n}$ will be the same, $a_{d,1}$ must be even.
    Since $a_{d,1} \equiv 1 \; \text{mod} \; 2m_1$ or $a_{d,1} \equiv m_1+1 \; \text{mod} \; 2m_1$, we conclude that $a_{d,1} \equiv m_1+1 \; \text{mod} \; 2m_1$.
    
    \textbf{Inductive Step:} Let $n$ be odd and $m=2^lm_1$ where $m_1$ odd. We may assume $l>1$ since we have proved the $l=1$ case. Assume $P_m(n)|d$ and $d>L_m(n)$.  Assume that there exists $z' \in \mathbb{Z}$ odd such that 
\[a_{d,s} \equiv 
\begin{cases}
z'm_1+1 \; \text{mod} \; 2^{l-1}m_1 & s=1\\
z'm_1 \; \text{mod} \; 2^{l-1}m_1  & 1 <s \leq n
\end{cases}.\]
Then we can take $t_s \in \mathbb{Z}^+$ such that $a_{d,s}=z'm_1+2^{l-1}m_1t_s$ for $1 <s \leq n$ and $t_1$ such that $a_{d,1}=z'm_1+1+2^{l-1}m_1t_1$.
By how $d$ was defined, $a_{2d,s} \equiv a_{d,s} \; \text{mod} \; m$. Theorem 5 from \cite{Paper1} tells us $a_{r,s}=\displaystyle{\sum_{i=1}^n a_{j,i}a_{r-j,s-i+1}}$, which we can use to compute $a_{2d,s}$ where $1<s \leq n$, and see that
\[a_{2d,s}=\sum_{i=1}^na_{d,i}a_{d,s-i+1}.\]
We now separate out the terms where $i,s-i+1=1$ to obtain
\[a_{d,1}a_{d,s}+a_{d,s}a_{d,1}+\sum_{\substack{i=2\\i \neq s}}^n a_{d,i}a_{d,s-i+1}\]
so we can write the $a_{d,s}$ in terms of $z',m_1,$ and $t_s$, which will be
\[2(z'm_1+1+2^{l-1}m_1t_1)(z'm_1+2^{l-1}m_1t_s)+\sum_{\substack{i=2\\i\neq s}}^n(z'm_1+2^{l-1}m_1t_i)(z'm_1+2^{l-1}m_1t_{s-i+1}).\]
Next, we multiply everything out:
\[2z'^2m_1^2+2^lz'm_1^2j_s+2z'm_1+2^lm_1t_s+2^lk'm_1^2j_1+2^{2l-1}m_1^2t_1t_s\]\[+\sum_{\substack{i=2\\i \neq s}}^n (z'^2m_1^2+2^{l-1}z'm_1^2t_{s-i+1}+2^{l-1}t_iz'm_1^2+2^{2l-2}m_1^2t_it_{s-i+1}).\]
Simplifying and reducing modulo $m$, this is equivalent to
\[ nz'^2m_1^2+2z'm_1+2^{l-1}z'm_1^2(2*\sum_{\substack{i=2\\ i \neq s}}^n t_i) \; \text{mod} \; m\]
or
\[m_1(nz'^2m_1+2z') \; \text{mod} \; m.\]
Take $z=nz'^2m_1+2z'$ and we have found a $z$ odd satisfying 
\[a_{d,s} \equiv zm_1 \; \text{mod} \; m.\]
Now $a_{2d,1} \equiv a_{d,1} \; \text{mod} \; m$ so we calculate $a_{2d,1}$:
\[a_{2d,1}=\sum_{i=1}^n a_{d,i}a_{2-i}.\]
Similar to before, we separate out the terms where $i, 2-i=1$ to see this is
\[a_{d,1}a_{d,1}+\sum_{i=2}^n a_{d,i}a_{d,2-i}\]
so we can plug in our other variables and have
\[(z'm_1+1+2^{l-1}m_1t_1)(z'm_1+1+2^{l-1}m_1t_1)+\sum_{i=2}^n(z'm_1+2^{l-1}m_1t_i)(z'm_1+2^{l-1}m_1t_{2-i}).\]
Expanding this results in
\[z'^2m_1^2+z'm_1+2^{l-1}z'm_1^2t_1+z'm_1+1+2^{l-1}m_1t_1+2^{l-1}z'm_1^2t_1+2^{l-1}m_1t_1+2^{2l-2}m_1t_1^2\]\[+\sum_{i=2}^n (z'^2m_1^2+2^{l-1}z'm_1^2t_{2-i}+2^{l-1}t_iz'm_1^2+2^{2l-2}m_1^2t_it_{2-i}),\]
which can simplify to 
\[ nz'^2m_1^2+2z'm_1+2^lm_1t_1+2^{l-1}z'm_1^2(2*\sum_{i=1}^n t_i)+1 \; \text{mod} \; m\]
or
\[ m_1(nz'^2m_1+2z')+1 \; \text{mod} \; m.\]
Notice that by how we defined $z$, this gives us $a_{d,1} \equiv zm_1+1 \; \text{mod} \; m$ and the lemma follows.
 \end{proof}
We now have the tools we will be using to prove our main theorem about the length when $n$ is odd:
 \begin{proof}[Proof of Theorem \ref{n_odd_length}]
 To prove $L_m(n)=l$, it suffices to show 
 \[D^l(0,0,...,0,1) \in K(\mathbb{Z}_m^n)\]
  and 
  \[D^{l-1}(0,0,...,0,1) \not \in K(\mathbb{Z}_m^n),\] or in other words,
 that $(a_{l,n}, a_{l,n-1}, ..., a_{l,1}) \in K(\mathbb{Z}_m^n)$ and $(a_{l-1,n}, a_{l-1,n-1}, ..., a_{l-1,1})$ is not. Because 
 \[(a_{l-1,n}, a_{l-1,n-1}, ..., a_{l-1,1})\] and 
 \[(2^{l-1}m_1+a_{l-1,n}, 2^{l-1}m_1+a_{l-1,n-1}, ..., 2^{l-1}m_1+a_{l-1,1})\] are both predecessors to $(a_{l,n}, a_{l,n-1}, ..., a_{l,1})$ by Lemma \ref{n odd m even pred}, showing that  
 \[(2^{l-1}m_1+a_{l-1,n}, 2^{l-1}m_1+a_{l-1,n-1}, ..., 2^{l-1}m_1+a_{l-1,1}) \in K(\mathbb{Z}_m^n)\]
  will mean that $(a_{l,n}, a_{l,n-1}, ..., a_{l,1}) \in K(\mathbb{Z}_m^n)$. Since $(a_{l,n}, a_{l,n-1}, ..., a_{l,1})$ cannot have more than one predecessor in $K(\mathbb{Z}_m^n)$, $(a_{l-1,n}, a_{l-1,n-1}, ..., a_{l-1,1})$ would not be in $K(\mathbb{Z}_m^n)$. Therefore, it suffices to show 
  \[(2^{l-1}m_1+a_{l-1,n}, 2^{l-1}m_1+a_{l-1,n-1}, ..., 2^{l-1}m_1+a_{l-1,1}) \in K(\mathbb{Z}_m^n).\] 
 
    \indent Take $d$ such that $P_m(n)|d$ and $d>L_m(n)$. Notice that if $1 \leq s \leq n$, then the $s$th entry of $D^d(2^{l-1}m_1+a_{l-1,n}, 2^{l-1}m_1+a_{l-1,n-1}, ..., 2^{l-1}m_1+a_{l-1,1})$ is 
    \[(2^{l-1}m_1+a_{l-1,n-s+1})a_{d,1}+(2^{l-1}m_1+a_{l-1,n-s})a_{d,2}+ \cdots+ (2^{l-1}m_1+a_{l-1,n-s+2})a_{d,n},\]
    which we aim to show is congruent to $2^{l-1}m_1+a_{l-1,n-s+1} \; \text{mod} \; m$. By Lemma \ref{coefficient_at_period}, there exists $z$ such that $a_{d,s} \equiv zm_1 \; \text{mod} \; m$ for $1 <s \leq n$ and $a_{d,1} \equiv zm_1+1 \; \text{mod} \; m$. So the $s$th entry is congruent to
    \[ (2^{l-1}m_1+a_{l-1,n-s+1})(zm_1+1)+(2^{l-1}m_1+a_{l-1,n-s})zm_1+ \cdots\]\[+ (2^{l-1}m_1+a_{l-1,n-s+2})zm_1 \; \text{mod} \; m.\]
    Expanding and reordering gives
    \[ 2^{l-1}m_1+a_{l-1,n-s+1}+zm_1\sum_{i=1}^n (2^{l-1}m_1+a_{l-1,i}) \; \text{mod} \; m.\]
    Focusing on the sum gives us
    \[2^{l-1}m_1+a_{l-1,n-s+1}+2^{l-1}zm_1^2n+km_1\sum_{i=1}^na_{l-1,i} \; \text{mod} \; m.\]
    Using Lemma \ref{coefficients_sum_to_power_of_2}, we have
    \[ 2^{l-1}m_1+a_{l-1,n-s+1}+2^{l-1}zm_1^2n+2^{l-1}km_1 \; \text{mod} \; m.\]
    Now we can pull out some common terms to obtain
    \[ 2^{l-1}m_1+a_{l-1,n-s+1} +2^{l-1}m_1z(m_1n+1)\; \text{mod} \; m.\]
    Because $m_1, n$ are odd, $m_1n+1$ is even, so this is
    \[ 2^{l-1}m_1+a_{l-1,n-s+1} \; \text{mod} \; m.\]
   It follows that $(2^{l-1}m_1+a_{l-1,n}, 2^{l-1}m_1+a_{l-1,n-1}, ..., 2^{l-1}m_1+a_{l-1,1}) \in K(\mathbb{Z}_m^n)$.
 \end{proof}
To prove our second main theorem, we first prove Lemma \ref{AlltheWayOut}:
\begin{proof}[Proof of Lemma \ref{AlltheWayOut}]
Let $n$ be odd and $m=2^lm_1$ where $m_1$ is odd and $l \geq 1$.

      \indent Note $L_m(n)=l$ by Theorem \ref{n_odd_length}. 
        Suppose there exists $\mathbf{u}=(x_1, x_2, ..., x_n)$ such that $x_1+x_2+ \cdots +x_n$ is odd and $Len(\mathbf{u})<l$. Then $D^{l-1}(\mathbf{u}) \in K(\mathbb{Z}_m^n)$. Choose $d$ such that $P_m(n)|d$. Then $D^{l-1+d}(\mathbf{u})=D^{l-1}(\mathbf{u})$. 
If we isolate the first entry of both sides, these must be equivalent, so
       \[(a_{l-1,1}x_1+a_{l-1,2}x_2+ \cdots+ a_{l-1,n}x_n)a_{d,1}+(a_{l-1,1}x_2+a_{l-1,2}x_3+ \cdots +a_{l-1,n}x_1)a_{d,2}\]\[+\cdots +(a_{l-1,1}x_n+a_{l-1,2}x_1+\cdots +a_{l-1,n}x_{m-1})a_{d,n} \]\[\equiv (a_{l-1,1}x_1+a_{l-1,2}x_2+ \cdots + a_{l-1,n}x_n) \; \text{mod} \; m.\]
       Using $a_{d,1} \equiv zm_1+1 \; \text{mod} \; m$ and $a_{d,s} \equiv zm_1 \; \text{mod} \; m$ for some $z$ odd, $s \neq 1$, we have 
     \[(a_{l-1,1}x_1+a_{l-1,2}x_2+ \cdots+ a_{l-1,n}x_n)(zm_1+1)+(a_{l-1,1}x_2+a_{l-1,2}x_3+ \cdots +a_{l-1,n}x_1)zm_1\]\[+\cdots +(a_{l-1,1}x_n+a_{l-1,2}x_1+\cdots +a_{l-1,n}x_{m-1})zm_1 \]\[\equiv (a_{l-1,1}x_1+a_{l-1,2}x_2+ \cdots + a_{l-1,n}x_n) \; \text{mod} \; m,\]  
     and so
    \[zm_1(a_{l-1,1}+a_{l-1,2}+ \cdots +a_{l-1,n})(x_1+x_2+ \cdots+ x_n) \equiv 0 \; \text{mod} \; m.\]
    By Lemma \ref{coefficients_sum_to_power_of_2},
    \begin{equation}\label{Equation_alltheway}
    2^{l-1}zm_1(x_1+x_2+ \cdots+x_n) \equiv 0 \; \text{mod} \; m.
    \end{equation}
    where Equivalency (\ref{Equation_alltheway}) follows from Lemma \ref{coefficients_sum_to_power_of_2}. Since $z, m_1,$ and $x_1+x_2+ \cdots +x_n$ are odd, Equivalency (\ref{Equation_alltheway}) is a contradiction. Therefore $Len(\mathbf{u})=l$.
\end{proof}
Finally, we can prove our last theorem of the paper:
 \begin{proof}[Proof of Theorem \ref{n_odd_Cycle_subgroup}]
 Let $n$ be odd and $m=2^lm_1$ where $l \geq 1$ and $m_1$ is odd. Let $G=\{(x_1, x_2, ..., x_n) \in \mathbb{Z}_m^n \; \mid \; x_1+x_2+ \cdots +x_n \equiv 0 \; \text{mod} \; 2^l\}$. 
 
  \indent Because of Theorem \ref{whenhaspred} and Lemma \ref{AlltheWayOut},   all of the tuples whose entries that sum up to be an odd number are on the outside of the connected component of its transition graph containing its Ducci cycle and all of the connected components go out exactly $l$ tuples from the cycle. Since every tuple that has a predecessor has exactly two predecessors, every tuple in the cycle has a total of $2^l$ tuples branching off from it, including itself, but not any other tuples in the cycle. We will use this knowledge about the transition graphs in the rest of the proof.
   
   \indent Let $\mathbf{v} \in K(\mathbb{Z}_m^n)$. Then, there is a tuple $\mathbf{u}=(x_1, x_2, ..., x_n) \in \mathbb{Z}_m^n$ such that $D^l(\mathbf{u})=\mathbf{v}$ and $x_1+x_2+ \cdots+ x_n$ is odd. So if $\mathbf{v}=(y_1, y_2, ..., y_n)$, then 
   \[y_i= a_{l,1}x_i+a_{l,2}x_{i+1}+ \cdots a_{l,n}x_{i-1}.\] Therefore $y_1+y_2+\cdots+y_n$ is
      \[\sum_{i=1}^n (a_{l,1}x_i+a_{l,2}x_{i+1} + \cdots + a_{l,n}x_{i-1})\]
      Expanding this sum and factoring leads us to
      \[(a_{l,1}+a_{l,2}+\cdots +a_{l,n})(x_1+x_2+\cdots+x_n).\]
      Using Lemma \ref{coefficients_sum_to_power_of_2}, this is
      \[2^l(x_1+x_2+\cdots+x_n),\]
     so $y_1+y_2+ \cdots + y_n \equiv 0 \; \text{mod} \; 2^l$.
      Therefore, $\mathbf{v} \in G$. Since $\mathbf{v} \in K(\mathbb{Z}_m^n)$ was arbitrary, $K(\mathbb{Z}_m^n) \leq G$. We now count up the elements in both of these groups. Since every tuple in $K(\mathbb{Z}_m^n)$ has $2^l$ tuples branching off of it, then,
      \[2^l|K(\mathbb{Z}_m^n)|=|\mathbb{Z}_m^n|=2^{nl}m_1^{n}\]
      
      which gives us $|K(\mathbb{Z}_m^n)|=2^{(n-1)l}m_1^{n}$. 
      
      \indent To find $|G|$, take $x_1, x_2, ..., x_{n-1}$ arbitrary. There are then $m_1$ choices for $x_n$ to make it so that $x_1+x_2+\cdots+x_n \equiv 0 \; \text{mod} \; 2^l$. Since $x_1, x_2, ..., x_{n-1}$ was arbitrary, this means that there are $2^{(n-1)l}m_1^{n-1}*m_1=2^{(n-1)l}m_1^{n}$ tuples in $|G|$. Therefore $|K(\mathbb{Z}_m^n)|=|G|$, and $K(\mathbb{Z}_m^n)=G$.
 \end{proof}


\begin{thebibliography}{99}
 \bibitem{Benjamin} Benjamin, A.T. \& Quinn, J.J. (2003). \textit{Proofs that Really Count: The Art of Combinatorial Proof}. Washington D.C. The Mathematical Association of America. \\
 \bibitem {Breuernote} Breuer, F. (1998). A Note on a Paper by Glaser and Sch\"{o}ffl. \textit{The Fibonacci Quarterly, 36(5)}, 463-466.\\   
   \bibitem {Breuer1} Breuer, F. (1999). Ducci Sequences Over Abelian Groups. \textit{Communications in Algebra, 27(12)}, 5999-6013.\\
   
  
   \bibitem {Breuer2} Breuer, F. (2010). Ducci Sequences and Cyclotomic Fields. \textit{Journal of Difference Equations and Applications, 16(7)}, 847-862.\\
   
   \bibitem{Brown} Brown, R. \& Merzel, J. (2007). The Length of Ducci's Four Number Game \textit{Rocky Mountain Journal of Mathematics, 37(1)}, 45-65.\\
    \bibitem{Cameron} Cameron, P. (1994). \textit{Combinatorics: Topics, Techniques, Algorithms}. Cambridge, England: Cambridge University Press.\\
    \bibitem{Chamberland} Chamberland, M. (2003). Unbounded Ducci Sequences. \textit{Journal of Difference Equations and Applications, 9(10)}, 887-895.\\
     \bibitem{Dular} Dular, B. (2020). Cycles of Sums of Integers. \textit{Fibonacci Quarterly, 58(2)}, 126-139.\\
     \bibitem{Ehrlich} Ehrlich, A. (1990). Periods in Ducci's $n$-Number Game of Differences. \textit{Fibonacci Quarterly, 28(4)}, 302-305. \\
       \bibitem{Freedman} Freedman, B (1948). The Four Number Game. \textit{Scripta Mathematica, 14}, 35-47.\\
      
       \bibitem {Glaser} Glaser, H. \& Sch\"{o}ffl, G. (1995). Ducci Sequences and Pascal's Triangle. \textit{Fibonacci Quarterly, 33(4)}, 313-324.\\
       \bibitem{Paper1} Lewis, M.L. \& Tefft, S.M. (2024). The Period of Ducci Cycles on $\mathbb{Z}_{2^l}$ for Tuples of Length $2^k$. Submitted for Publication. $<$ArXiv: 2401.17502$>$ \\
       \bibitem{Furno}
       Ludington Furno, A (1981). Cycles of differences of integers. \textit{Journal of Number Theory, 13(2)}, 255-261. \\
       
       
       \bibitem {Schinzel} Misiurewicz, M., \& Schinzel, A. (1988). On $n$ Numbers in a Circle. \textit{Hardy Ramanujan Journal, 11}, 30-39.\\
       \bibitem{Wong} Wong, F.B. (1982). Ducci Processes. \textit{The Fibonacci Quarterly, 20(2)}, 97-105.
      
   \end{thebibliography}
\end{document}